\documentclass[11pt,twoside]{article}
\usepackage{etex}
\usepackage[dvipsnames]{xcolor} 
\usepackage{color}

\usepackage{booktabs}
\usepackage{amsthm}

\usepackage{amsfonts,amssymb,amsmath,amsxtra,url,float} 
\allowdisplaybreaks[4] 
\usepackage[colorlinks,
  linkcolor=magenta, %
  anchorcolor=Periwinkle,
  citecolor=violet,
  urlcolor=blue
  ]{hyperref} 
\usepackage{enumitem}
\usepackage{geometry} 
\usepackage{rotating} 
\usepackage{lscape} 
\usepackage{yhmath} 
\usepackage{multirow}
\usepackage{graphicx} 
\usepackage{subfigure} 
\usepackage{tikz}
\usepackage{pgfplots}
\usepackage{tikz-3dplot}
\usetikzlibrary{patterns}
\usetikzlibrary{3d,calc}
\usetikzlibrary{decorations.pathreplacing,decorations.markings}
 \tikzset{
  on each segment/.style={
    decorate,
    decoration={
      show path construction,
      moveto code={},
      lineto code={
        \path [#1]
        (\tikzinputsegmentfirst) -- (\tikzinputsegmentlast);
      },
      curveto code={
        \path [#1] (\tikzinputsegmentfirst)
        .. controls
        (\tikzinputsegmentsupporta) and (\tikzinputsegmentsupportb)
        ..
        (\tikzinputsegmentlast);
      },
      closepath code={
        \path [#1]
        (\tikzinputsegmentfirst) -- (\tikzinputsegmentlast);
      },
    },
  },
  mid arrow/.style={postaction={decorate,decoration={
        markings,
        mark=at position 0.6 with {\arrow[#1]{stealth}} 
      }}},
}
\usetikzlibrary{arrows}
\usetikzlibrary{trees}

\usetikzlibrary{matrix}
\usetikzlibrary{patterns}
\usetikzlibrary{shadings} 
\usepackage{fancyhdr} 
\def\headertitle{Geometric Realizations of Capability Spaces of Large Models and Zigzag Persistent Homology}
\def\fstpage{1} 
\def\page{$\begin{matrix} {\color{white}0} \\ \thepage \end{matrix}$} 

\usepackage[all]{xy} 
\usepackage{dsfont} 
\usepackage{cite}
\usepackage{mathrsfs} 
\numberwithin{figure}{section}
\usepackage{marginnote} 
\usepackage{graphicx} 
\usepackage{multicol} 

\usepackage{enumitem}
\setenumerate[1]{itemsep=0pt,partopsep=0pt,parsep=\parskip,topsep=3pt}
\setitemize[1]{itemsep=0pt,partopsep=0pt,parsep=\parskip,topsep=3pt}
\setdescription{itemsep=0pt,partopsep=0pt,parsep=\parskip,topsep=3pt}
\setlist[itemize]{leftmargin=35pt}
\setlist[enumerate]{leftmargin=35pt}
\usepackage{changepage} 
\newcommand{\checks}[1]{{\color{red}{#1}}} 

\usepackage{contour} 
\usepackage{xcolor}
\contourlength{0.03em}
\contournumber{auto}

\newtheorem{theorem}{Theorem}[section]
\newtheorem{lemma}[theorem]{Lemma}
\newtheorem{corollary}[theorem]{Corollary}
\newtheorem{main theorem}[theorem]{Main Theorem}
\newtheorem{proposition}[theorem]{Proposition}
\newtheorem{definition}[theorem]{Definition}

\newtheorem{remark}[theorem]{Remark}
\newtheorem{example}[theorem]{Example}

\numberwithin{equation}{section}

\def\orcid{
\begin{tikzpicture}[baseline=-1mm]
\filldraw[Green!35] (0,0) circle (5pt);
\filldraw[white] (0,0) node{\tiny\textbf{iD}};
\end{tikzpicture}
}

\usepackage{extarrows}

\def\heart{{\color{red}\pmb{\heartsuit}}}
\def\diamo{{\color{red}\pmb{\diamondsuit}}}

\def\orcid{
\begin{tikzpicture}[baseline=-1mm]
\filldraw[Green!35] (0,0) circle (5pt);
\filldraw[white] (0,0) node{\tiny\textbf{iD}};
\end{tikzpicture}
}

\title{\bf
Geometric Realizations of Capability Spaces of Large Models and Zigzag Persistent Homology
}

\vspace{5mm}

\author{
Yu-Zhe Liu $^{\ref{Author1}, \ref{Author4}, \href{https://orcid.org/0009-0005-1110-386X}{\orcid}\ref{orcid1}}$,~
Shengda Liu $^{\ref{Author2},{\color{red}\ddag}, \href{https://orcid.org/0000-0003-1382-4212}{\orcid}\ref{orcid2}}$,~
Jianyong Pi $^{\ref{Author3}}$,~
Chao Zhang $^{\ref{Author1}, \ref{Author4}}$
}
\date{ }
\begin{document}






\maketitle

\begin{enumerate}[label=\textbf{\color{red}\arabic*}] \footnotesize
  \item
    \begin{center}
      School of Mathematics and Statistics,
      Guizhou University, Guiyang 550025, Guizhou, China;
    \end{center}
    \label{Author1}

  \item
    \begin{center}
      The Key Laboratory of Cognition and Decision Intelligence for Complex Systems,
      Institute of Automation,
      Chinese Academy of Sciences, Beijing 100190, China;
    \end{center}
    \label{Author2}

  \item
    \begin{center}
      School of Computer Science and Technology, Guizhou University, Guiyang 550025, Guizhou, China;
    \end{center}
    \label{Author3}

  \item
    \begin{center}
      Guizhou Provincial Key Laboratory of Applied Mathematics and Computing Power \& Algorithms,
      Guiyang, 550025, Guizhou, China;
    \end{center}
    \label{Author4}

  \item[]
    \begin{center}
      E-mail:
      \url{liuyz@gzu.edu.cn} / \url{yzliu3@163.com} (Y.-Z. Liu) \par
      \url{thinksheng@foxmail.com} (S. Liu)
      \url{jypi@gzu.edu.cn} (J. Pi)
      \url{zhangc@amss.ac.cn} (C. Zhang)
    \end{center}
\end{enumerate}
\vspace{1mm}
\begin{enumerate}[label=\textbf{\color{red}$\ddag$}]
  \item \footnotesize
    \begin{center}
      Corresponding author
    \end{center} \label{CorrespondingAuthor}
\end{enumerate}
\vspace{1mm}
\begin{enumerate}[leftmargin=6.5cm] \footnotesize
  \item[\orcid]
      ORCID: \href{https://orcid.org/0009-0005-1110-386X}{0009-0005-1110-386X}
      \label{orcid1} 
  \item[\orcid]
      ORCID: \href{https://orcid.org/0009-0005-1110-386X}{0000-0003-1382-4212}
      \label{orcid2} 
\end{enumerate}

\hspace{2mm}



\begin{adjustwidth}{1cm}{1cm} 
\noindent \footnotesize 
\textbf{Abstract.}
We propose a topological model for capability changes during large-model training.
The model represents knowledge states and preregistered capability probes as subcomplexes of a finite regular CW complex.
We prove that attaching a single $n$-cell can only create a class in $\mathrm H_n$ or kill a class in $\mathrm H_{n-1}$.
With the reliability threshold and the training checkpoint as two parameter axes, capability gains and losses make the spaces along the training axis nonnested.
Union or intersection bridges produce zigzag persistence modules, and we distinguish checkpoint, transition, and bridge-sensitive classes. A finite four-stage example is computed by boundary matrices.
The framework records algebraic changes under a chosen encoding, and a homological change is called a candidate capability emergence only after robustness tests, null-model comparisons, and independent behavioral validation.



\end{adjustwidth}

%
%
%




\setcounter{tocdepth}{2}
\setcounter{secnumdepth}{3}
\tableofcontents


\def\gldim{\mathrm{gl.dim}}
\def\perm{\mathrm{perm}}
\def\forb{\mathrm{forb}}
\def\rmv{\mathsf{v}}
\def\rmp{\mathsf{w}}
\def\Pic{Figure~}


\section{Introduction}

\def\kk{\Bbbk}
\def\CC{\mathbb{C}}
\def\FF{\mathbb{F}}
\def\PP{\mathbb{P}}
\def\ZZ{\mathbb{Z}}
\def\>={\geqslant}
\def\=<{\leqslant}

\def\HHomolo{\mathrm{HH}}
\def\Homolo{\mathrm{H}}

\def\calA{\mathcal{A}}
\def\calB{\mathcal{B}}
\def\calC{\mathcal{C}}
\def\calD{\mathcal{D}}
\def\calE{\mathcal{E}}
\def\calF{\mathcal{F}}
\def\calG{\mathcal{G}}
\def\calH{\mathcal{H}}
\def\calI{\mathcal{I}}
\def\calJ{\mathcal{J}}
\def\calK{\mathcal{K}}
\def\calL{\mathcal{L}}
\def\calM{\mathcal{M}}
\def\calN{\mathcal{N}}
\def\calO{\mathcal{O}}
\def\calP{\mathcal{P}}
\def\calQ{\mathcal{Q}}
\def\calR{\mathcal{R}}
\def\calS{\mathcal{S}}
\def\calT{\mathcal{T}}
\def\calU{\mathcal{U}}
\def\calV{\mathcal{V}}
\def\calW{\mathcal{W}}
\def\calX{\mathcal{X}}
\def\calY{\mathcal{Y}}
\def\calZ{\mathcal{Z}}

\def\bfP{\mathbf{P}}
\newcommand{\bfp}{\boldsymbol p}

\def\Hom{\mathrm{Hom}}
\def\End{\mathrm{End}}
\def\Ext{\mathrm{Ext}}
\def\Ker{\mathrm{Ker}}
\def\Ima{\mathrm{Im}}
\def\Picagroup{\mathscr{P}\!ic}

\def\Tot{\mathrm{Tot}}
\def\divisor{\mathrm{div}}
\def\Div{\mathrm{Div}}
\def\Aff{\mathbf{A}}

\def\Ring{\mathsf{Ring}}
\def\Modcat{\mathsf{Mod}}
\def\Coh{\mathsf{Coh}}
\def\Dcat{\mathsf{D}}

\newcommand{\defines}[1]{{\color{blue!75}\it #1}}
\def\Cell{\operatorname{Cell}}
\def\rank{\operatorname{rank}}

When a model reads a token sequence, each position can attend to information at other positions in the sequence and thereby determine the current output.
This process involves a core mechanism called self-attention, which is also the core mechanism of Transformer language models \cite{vaswani2017}.
These models learn parameterized conditional distributions from prediction objectives on large corpora.
Intuitively, a large model performs prediction on massive amounts of text.
The most common task is to predict the next word given the preceding words.
Thus, a large model learns how to infer the probability distribution of the next token from the available context.
This distribution is determined by a large number of parameters.
Training a large model means continually adjusting these parameters so that its predictions become more accurate.
It has been observed that within the range covered by experiments, the loss usually decreases according to a stable pattern as the model size, data size, and computation increase.
This pattern is not arbitrary but approximately a power law, known as empirical scaling laws.
This suggests that large-model performance is predictable to some extent \cite{kaplan2020}.
Under a fixed compute budget, one cannot simply make the model larger.
The number of training tokens also matters.
Model size and training data size need to be matched.
For example, a very large model trained on too few tokens may perform worse than a smaller model trained more thoroughly (see compute-optimal training, \cite{hoffmann2022}).
Some capabilities appear suddenly under discrete evaluation metrics.
They are called emergent abilities of large language models.
Specifically, on some tasks, small models perform close to random, and after the model size increases beyond a certain point, accuracy suddenly improves markedly.
It appears that the model suddenly ``learns'' a certain capability.
Such phenomena are called emergent abilities in \cite{wei2022}.
However, this interpretation remains controversial.
Nonlinear or discontinuous metrics can create apparent jumps.
Therefore, some emergence curves become smooth under continuous metrics.
Some researchers argue that the so-called ``sudden appearance'' may not reflect a real internal change in the model but may be an artifact of the evaluation metric.
For example, a metric such as exact match is binary.
If the model output improves slightly, the score may jump directly from 0 to 1, which looks like sudden learning.
If one uses a continuous metric, such as token-level accuracy or probability score, the curve may become smooth without an obvious jump.
Schaeffer et al. proposed this view in 2023 \cite{schaeffer2023}.

We replace the question whether one score crosses one threshold with a structural question.
We ask whether the organization among capabilities changes.
We first realize a knowledge or capability system as a geometric object.
The content mastered at one stage becomes a subspace of this object.
Ordinary persistent homology is suitable for a monotonically expanding knowledge state.
Large models may gain and lose capabilities during continued pretraining, instruction tuning, and domain adaptation.
Catastrophic forgetting has been observed for some models and task orders \cite{luo2023}.
When measured capability spaces (see Definition \ref{def:capab-space}) are not nested, they form a zigzag diagram with arrows in both directions \cite{carlsson2010}.
In the language of topological data analysis, the measured scores or hidden representations are the data, the associated complexes are geometric summaries, and persistent homology records how their homology changes across a scale or a threshold \cite{edelsbrunner2010}. The checkpoint (see Definition \ref{def:training-trajectory}) direction requires zigzag persistence when these summaries are not nested.
This distinction between the threshold direction and the checkpoint direction is illustrated in \Pic \ref{fig:cks mod}
(this figure shows two persistence directions of the capability-space construction:
at a fixed checkpoint, decreasing the reliability threshold produces an ordinary filtration;
at a fixed threshold, varying the training checkpoint produces a nonnested sequence,
and union bridges convert this sequence into a zigzag diagram).

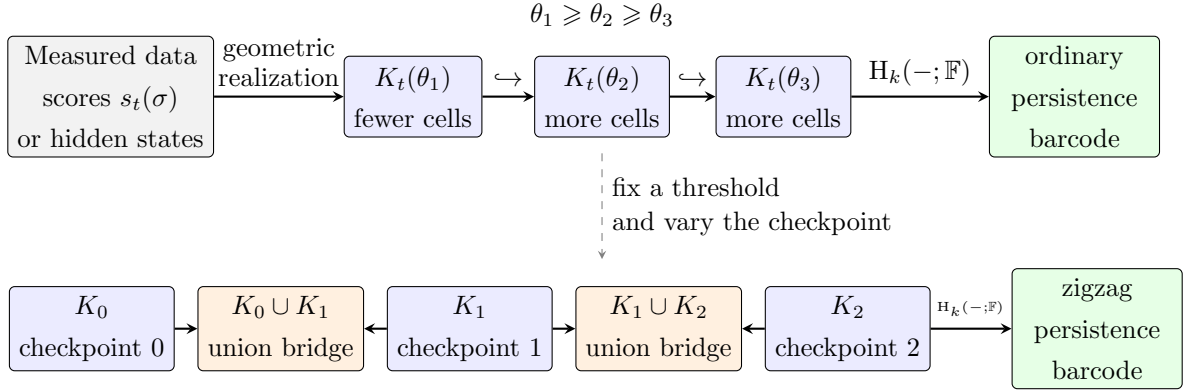
\begin{figure}[htbp]
  \centering
\begin{tikzpicture}[
  x=1cm,
  y=1cm,
  >=stealth,
  font=\small,
  data/.style={
    draw,
    rounded corners=2pt,
    fill=gray!10,
    minimum width=2.4cm,
    minimum height=1.05cm,
    align=center
  },
  state/.style={
    draw,
    rounded corners=2pt,
    fill=blue!8,
    minimum width=1.75cm,
    minimum height=1.05cm,
    align=center
  },
  bridge/.style={
    draw,
    rounded corners=2pt,
    fill=orange!12,
    minimum width=2.15cm,
    minimum height=1.05cm,
    align=center
  },
  output/.style={
    draw,
    rounded corners=2pt,
    fill=green!10,
    minimum width=2.25cm,
    minimum height=1.05cm,
    align=center
  },
  ordinary/.style={
    ->,
    line width=0.8pt
  },
  zigzag/.style={
    ->,
    line width=0.8pt
  }
]



\node[data] (data) at (0.25,3.27)
  {Measured data\\
   scores $s_t(\sigma)$\\
   or hidden states};

\node[state] (kt1) at (4.25,3.27)
  {$K_t(\theta_1)$\\
   fewer cells};

\node[state] (kt2) at (6.75,3.27)
  {$K_t(\theta_2)$\\
   more cells};

\node[state] (kt3) at (9.15,3.27)
  {$K_t(\theta_3)$\\
   more cells};

\node[output] (ordinarybarcode) at (13,3.27)
  {ordinary\\
   persistence\\ barcode};

\draw[ordinary]
  (data) -- node[above,align=center]
  {geometric \\[-1.5mm] realization} (kt1);

\draw[ordinary]
  (kt1) -- node[above] {$\hookrightarrow$} (kt2);

\draw[ordinary]
  (kt2) -- node[above] {$\hookrightarrow$} (kt3);

\draw[ordinary]
  (kt3) -- node[above]
  {$\mathrm H_k(-;\FF)$} (ordinarybarcode);

\node at (6.75,4.35)
  {$\theta_1\geqslant\theta_2\geqslant\theta_3$};



\node[state] (k0) at (0,0.20)
  {$K_0$\\
   checkpoint $0$};

\node[bridge] (u01) at (2.5,0.20)
  {$K_0\cup K_1$\\
   union bridge};

\node[state] (k1) at (5,0.20)
  {$K_1$\\
   checkpoint $1$};

\node[bridge] (u12) at (7.5,0.20)
  {$K_1\cup K_2$\\
   union bridge};

\node[state] (k2) at (10,0.20)
  {$K_2$\\
   checkpoint $2$};

\node[output] (zigzagbarcode) at (13.3,0.20)
  {zigzag\\
   persistence\\ barcode};

\draw[zigzag] (k0) -- (u01);
\draw[zigzag] (k1) -- (u01);
\draw[zigzag] (k1) -- (u12);
\draw[zigzag] (k2) -- (u12);

\draw[ordinary]
  (k2) -- node[above]
  {\tiny $^{\mathrm H_k(-;\FF)}$} (zigzagbarcode);



\draw[dashed,gray,->]
  (6.75,2.52) --
  node[right,align=left,text=black]
  {fix a threshold\\[-1mm]and vary the checkpoint}
  (6.75,1.12);

\end{tikzpicture}
\caption{Ordinary persistence across reliability thresholds and zigzag persistence across training checkpoints}
\label{fig:cks mod}
\end{figure}

We use the set-theoretic view of knowledge space theory \cite{doignon1985,doignon1999}.
We introduce a \emph{cellular knowledge-state model}, abbreviated to the CKS model.
This model is not a standard part of classical knowledge space theory.
It places knowledge items, inference paths, and higher-order synthesis in one geometric object.
Previous work studied knowledge spaces by pretopology, decomposition, and gluing \cite{danilov2009}.
Our questions about cell attachment and homology are different.
We then extend the CKS model to capability spaces of large models.
Zigzag persistent homology describes acquisition, preservation, forgetting, recovery, and structural replacement.
Our results provide a modeling framework and testable statements.
We do not identify every homology class with a semantic capability.
We also do not interpret topological correlation as a causal mechanism.

\section{Preliminaries}\label{sec:preliminaries}

\textsl{This section follows the dependence of the mathematical concepts.
We first recall simplices and probability simplices.
We then introduce homology, persistence modules, and barcodes.
The last subsection defines the terms from large models.
The remarks provide intuitive explanations for readers from the other field.}

\subsection{Simplices and geometric models}

Let $v_0,\ldots,v_n\in\mathbb R^d$ be affinely independent.
The \defines{$n$-simplex} spanned by these points is
\[
[v_0,\ldots,v_n]=\left\{\sum_{i=0}^n a_iv_i\ \middle|\ a_i\geqslant0,\ \sum_{i=0}^na_i=1\right\}.
\]
A $k$-simplex spanned by $k+1$ of these vertices is called a $k$-face.
Obviously,
a $0$-simplex is a point,
a $1$-simplex is a closed line segment,
a $2$-simplex is a filled triangle,
and a $3$-simplex is a filled tetrahedron.
The coefficients $a_i$ are the weights of a point relative to the vertices.
Their nonnegativity and their sum place the point inside the convex region.

Let $E=\{e_1,\ldots,e_m\}$ be a finite set.
The \defines{probability simplex} on $E$ is
\[
\Delta(E)=\left\{(p_1,\ldots,p_m)\in\mathbb R^m\ \middle|\ p_i\geqslant0,\ \sum_{i=1}^m p_i=1\right\}.
\]
The vector $(p_1,\ldots,p_m)$ represents the probability distribution with mass $p_i$ at $e_i$.
For two outcomes, every distribution has the form $(p,1-p)$.
These distributions form a line segment.
For three outcomes, they form a filled triangle.
In general, distributions on $m$ outcomes form an $(m-1)$-simplex.
A vertex means that one outcome has probability one.
An interior point assigns positive probability to several outcomes.
The next-token distribution of a language model is a point in such a simplex.

\begin{definition}\label{def:simp-complex}\rm
A finite family $K$ of simplices is a \defines{simplicial complex} if the following conditions hold.
\begin{itemize}
\item Every face of a simplex in $K$ belongs to $K$.
\item For $\sigma,\tau\in K$, the intersection $\sigma\cap\tau$ is empty or is a common face.
\end{itemize}
The union $|K|=\bigcup\limits_{\sigma\in K}\sigma$ is the geometric realization of $K$.
\end{definition}

A simplicial complex glues points, segments, filled triangles, and higher simplices along common faces.
If a filled triangle is present, its edges and vertices must also be present.
This rule supports uniform computation.
It may be too restrictive when a whole relation and its subrelations have different meanings.
For this reason, we also use CW complexes in this paper.

\begin{definition}\label{def:cw-complex}\rm
A \defines{finite CW complex} $K$ is a space with a finite filtration
\[ \varnothing=K^{(-1)}\subseteq K^{(0)}\subseteq\cdots\subseteq K^{(N)}=K. \]
The space $K^{(0)}$ is a finite discrete set.
For each $n\geqslant 1$, the space $K^{(n)}$ is obtained from $K^{(n-1)}$ by attaching finitely many open $n$-balls.
An $n$-cell uses a closed ball $\mathbb{D}^n$ and a continuous attaching map
\[ f:\mathbb{S}^{n-1}=\partial \mathbb{D}^n \longrightarrow K^{(n-1)}. \]
The boundary is identified with its image under $f$.
The interior becomes an open $n$-cell.
The complex is \defines{regular} if the closure of every cell is homeomorphic to a closed ball and its boundary is a union of lower-dimensional cells \cite{hatcher2002}.
\end{definition}

The letter C refers to closure finiteness, and the letter W refers to the weak topology.
A finite CW complex can be built dimension by dimension.
To be precise, we first place points, and then we attach segments by their endpoints.
Next, we attach disks by their boundary circles.
A single disk cell can fill a loop made of many edges.
No triangulation of the disk is required.
Every finite simplicial complex is a regular CW complex.

\subsection{Persistent homology}

\begin{definition}\label{def:chains}\rm
Choose orientations for the cells of a finite CW complex $K$,  and fix a coefficient field $\FF$.
The cellular chain group $C_k(K;\FF)$ is the $\FF$-vector space with the $k$-cells as a basis.
The oriented cellular boundary extends linearly to $\partial_k: C_k(K;\FF)\to C_{k-1}(K;\FF)$.
We set
\[ Z_k(K)=\ker\partial_k, \quad B_k(K)=\operatorname{im}\partial_{k+1}.\]
Elements of $Z_k(K)$ and $B_k(K)$ are called \defines{$k$-cycles} and \defines{$k$-boundaries}, respectively.
We write $U\le V$ when $U$ is a vector subspace of $V$ in this paper.
\end{definition}

\begin{lemma}\label{lem:boundary-square}
For every $k$, we have $\partial_k\partial_{k+1}=0$.
It follows that $B_k(K)\le Z_k(K)$.
\end{lemma}

This lemma is well-known and it can be proved by direct computation.
Indeed, consider the oriented simplex $[v_0,\ldots,v_{k+1}]$. Its boundary is
\[ \partial_{k+1}[v_0,\ldots,v_{k+1}]=\sum_{i=0}^{k+1}(-1)^i[v_0,\ldots,\widehat v_i,\ldots,v_{k+1}]. \]
Each codimension-two face appears twice after a second boundary is taken.
Assume $i<j$. Deleting $v_i$ before $v_j$ has coefficient $(-1)^i(-1)^{j-1}$,
and deleting $v_j$ before $v_i$ has coefficient $(-1)^j(-1)^i$.
The two coefficients have opposite signs. Thus all terms cancel.
The cellular boundary is defined by degrees of attaching maps and satisfies the same identity \cite{hatcher2002}.
Moreover,  take $b\in B_k(K)$. By Definition \ref{def:chains}, there is $c\in C_{k+1}(K;\FF)$ such that $b=\partial_{k+1}c$.
Thus $\partial_kb=\partial_k\partial_{k+1}c=0$. Hence $b\in Z_k(K)$.

\begin{definition}\label{def:homology}\rm
Lemma~\ref{lem:boundary-square} shows that the quotient $Z_k(K)/B_k(K)$ is defined.
The $k$th homology group and the $k$th Betti number are
\[ \mathrm H_k(K;\FF)=Z_k(K)/B_k(K),\quad \beta_k(K)=\dim_{\FF}\mathrm H_k(K;\FF). \]
Two cycles represent the same homology class if and only if their difference is the boundary of a chain of one higher dimension.
\end{definition}

For a geometric object $K$ considered here,
$\mathrm H_0(K;\FF)$ records connected components,
$\mathrm H_1(K;\FF)$ records closed loops which are not filled by two-dimensional chains,
and $\mathrm H_2(K;\FF)$ records closed two-dimensional shells which are not filled by three-dimensional chains.
In general, $\mathrm H_n(K;\FF)$ detects $n$-cycles with zero boundary which are not boundaries of higher-dimensional chains \cite{hatcher2002,edelsbrunner2010}. Cf. \Pic \ref{fig:low-homology}.

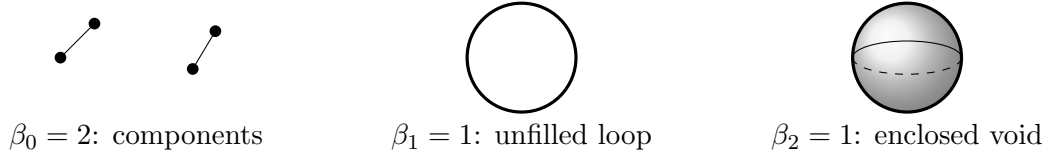
\begin{figure}[htbp]
\centering
\begin{tikzpicture}
\begin{scope}[xshift=0cm]
\fill (-1.00, 0.00) circle (2.2pt);
\fill (-0.55, 0.45) circle (2.2pt);
\fill ( 0.75,-0.15) circle (2.2pt);
\fill ( 1.05, 0.35) circle (2.2pt);
\draw (-1.00, 0.00)--(-0.55, 0.45);
\draw ( 0.75,-0.15)--( 1.05, 0.35);
\node at (0,-1.05) {$\beta_0=2$: components};
\end{scope}
\begin{scope}[xshift=5.1cm]
\draw[very thick] (0,0) circle (0.72); \node at (0,-1.05) {$\beta_1=1$: unfilled loop};
\end{scope}
\begin{scope}[xshift=10.2cm]
\shade[ball color=gray!20,opacity=.75] (0,0) circle (0.72);
\draw[very thick] (0,0) circle (0.72);
\draw[dashed] (-0.72,0) arc (180:360:0.72 and 0.22);
\draw (0.72,0) arc (0:180:0.72 and 0.22);
\node at (0,-1.05) {$\beta_2=1$: enclosed void};
\end{scope}
\end{tikzpicture}
\caption{Geometric intuition for low-dimensional homology;
the precise objects are the quotient spaces in Definition \ref{def:homology}}
\label{fig:low-homology}
\end{figure}

\begin{proposition}\label{prop:homology-functoriality}
Let $f:K\to L$ be a cellular map, and $f_{\bullet}:C_k(K;\FF)\to C_k(L;\FF)$ be its induced chain map.
Then $f$ induces a well-defined linear map
\[ \mathrm H_k(f):\mathrm H_k(K;\FF) \longrightarrow \mathrm H_k(L;\FF),\qquad [z]\longmapsto[f_{\bullet}z]. \]
The identity map induces the identity homomorphism.
Moreover, $\mathrm H_k(g\circ f)=\mathrm H_k(g)\circ\mathrm H_k(f)$.
\end{proposition}

Indeed, the chain-map identity is $\partial f_{\bullet}=f_{\bullet}\partial$.
If $z\in Z_k(K)$, then $\partial(f_{\bullet}z)=f_{\bullet}(\partial z)=0$.
Thus $f_{\bullet}z\in Z_k(L)$.
If $[z]=[z']$, Definition~\ref{def:homology} provides $c\in C_{k+1}(K;\FF)$ with $z-z'=\partial c$.
It follows that $f_{\bullet}z-f_{\bullet}z'=\partial(f_{\bullet}c)\in B_k(L)$.
Then $[f_{\bullet}z]=[f_{\bullet}z']$.
Linearity and compatibility with identities and compositions follow from the corresponding properties of chain maps.

\begin{definition}\label{def:persistence-module}\rm
A \defines{finite filtration} is a nested sequence of subcomplexes
\[ K_0\hookrightarrow K_1\hookrightarrow\cdots\hookrightarrow K_T. \]
Proposition~\ref{prop:homology-functoriality} provides linear maps and the $k$th \defines{persistence module}
\[
\mathrm H_k(K_0)\longrightarrow\mathrm H_k(K_1)\longrightarrow\cdots\longrightarrow\mathrm H_k(K_T).
\]
Persistent homology studies the birth, continuation, and death of homology classes in this module.
\end{definition}

\begin{definition}\label{def:interval-module}\rm
For an integer interval $[a,b]\subseteq\{0,\ldots,T\}$, the \defines{interval module} $I[a,b]$ is $\FF$ at every index in $[a,b]$ and zero elsewhere.
Every structure map inside the interval is the identity.
The multiset of intervals in an interval decomposition is called a \defines{barcode}.
Each horizontal bar records the index range of one class.
Its left endpoint is the birth index, and its right endpoint is the death index.
\end{definition}

\begin{example}\label{ex:point-cloud}\rm
Recall that a \defines{point cloud} is a finite set in a Euclidean space, and
it records coordinates without prescribing edges or faces.
Take ten points $P=\{p_1,\ldots,p_{10}\}\subset\mathbb R^2$ near a circle.
For $r \geqslant 0$, set
\[ U_r = \bigcup_{i=1}^{10}\overline B(p_i,r). \]
Thus, every point is covered by a closed disk of radius $r$.
It is clear that $U_r\subseteq U_s$ holds for all $r\leqslant s$.
The increasing radius provides a filtration, and the corresponding \v{C}ech complex $\check C_r(P)$ has the points as vertices.
A simplex is inserted when the associated disks have a common point.
All disks and their finite intersections are convex.
The nerve theorem shows that $|\check C_r(P)|\simeq U_r$ \cite{edelsbrunner2010}.
Figure \ref{fig:point-cloud-filtration} shows three stages.
For small $r$, there are ten connected components.
When adjacent disks overlap around the circle, there is one component and one central hole.
For larger $r$, the central hole is filled.
Nine $\mathrm H_0(\check C_r(P))$ classes die during merging,
one $\mathrm H_0(\check C_r(P))$ class survives;
and one $\mathrm H_1(\check C_r(P))$ class is born when the ring forms and dies when the central hole is filled.
\end{example}

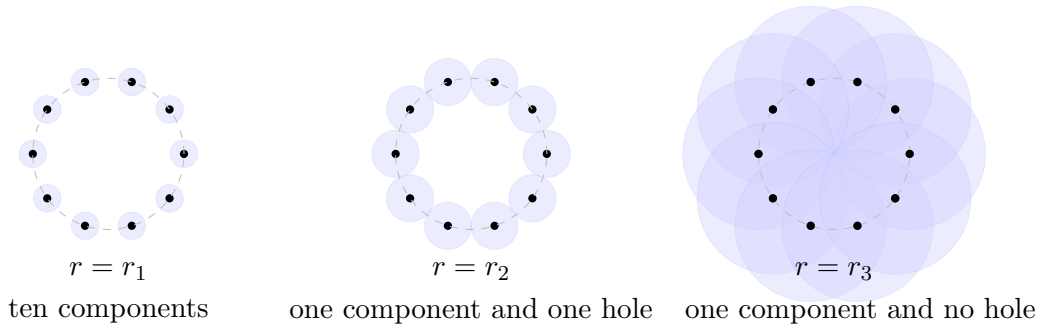
\begin{figure}[htbp]
\centering
\begin{tikzpicture}
\foreach \shift/\rad/\lab in {0/0.18/$r=r_1$, 4.8/0.31/$r=r_2$, 9.6/1.00/$r=r_3$}
{
  \begin{scope}[xshift=\shift cm]
  \foreach \a in {0,36,...,324}
  {
    \filldraw[fill=blue!18,draw=blue!25, opacity=0.4]
    ({cos(\a)},{sin(\a)}) circle (\rad);
    \fill ({cos(\a)},{sin(\a)}) circle (1.55pt);
  }
  \draw[gray!50,dashed] (0,0) circle (1);
  \node at (0,-1.55) {\lab};
  \end{scope}
}
\foreach \x in {0,36,...,324}
{
  \fill[shift={(9.60,0)}][rotate around = {\x:(0,0)}] (1,0) circle (1.55pt);
}
\node at ( 0.00,-2.08) {ten components};
\node at ( 4.80,-2.08) {one component and one hole};
\node at ( 9.95,-2.08) {one component and no hole};
\end{tikzpicture}
\caption{Disk expansion for a ten-point cloud, where the radii are schematic.}
\label{fig:point-cloud-filtration}
\end{figure}

\begin{figure}[htbp]
\centering
\begin{tikzpicture}
\draw[->] (0,0)--(9,0) node[right] {$r$};
\draw (0,0) node{$|$} node[below] {$0$};
\draw (2,0) node{$|$} node[below] {$r_2$};
\draw (8,0) node{$|$} node[below] {$r_3$};
\foreach \x in {0.2, 0.4, 0.6, 0.8, 1.0, 1.2, 1.4, 1.6, 1.8}
\draw[black][shift={(0, \x)}][line width =  1pt] (0,0)--(2,0);
\draw[black][shift={(0,2.0)}][line width =  1pt] (0,0)--(9,0);
\draw[black][shift={(0,2.2)}][line width =1.5pt] (2,0)--(8,0) [red];
\draw[black] (0,1.0) node[left]{$\mathrm H_0$};
\draw[ red ] (2,2.2) node[above]{$\mathrm H_1$};
\end{tikzpicture}
\caption{A schematic barcode for Example \ref{ex:point-cloud}}
\label{fig:point-cloud-barcode}
\end{figure}
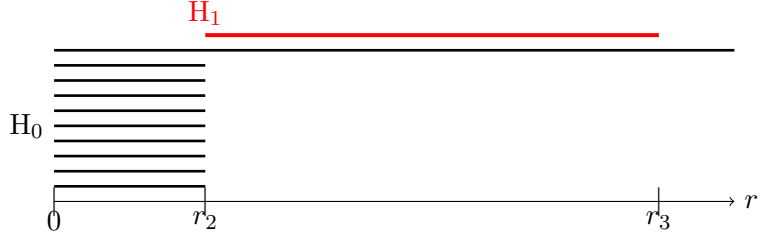

\begin{remark}\rm
Topological data analysis, abbreviated to TDA, studies data through geometric or combinatorial objects and algebraic invariants of these objects \cite{edelsbrunner2010}. A persistence-based TDA construction has four parts. One first specifies the observed data. One then constructs a filtered family of spaces or complexes. Homology converts the filtration into a persistence module. The interval decomposition in Theorem~\ref{thm:persistence-decomposition} produces a barcode.
Example~\ref{ex:point-cloud} follows this pipeline. The observed data are the coordinates in $P$. The radius $r$ constructs the nested family $\{\check C_r(P)\}_{r\geqslant0}$. The homology groups record components and holes. Figure~\ref{fig:point-cloud-barcode} summarizes their persistence. The construction rule from data to a filtration is part of the model. The barcode is therefore interpreted relative to that rule.
\end{remark}

We represent the vector spaces and structure maps from Definition~\ref{def:persistence-module} as matrices. Successive basis changes that preserve the index order and commute with the structure maps track each basis vector through a consecutive range of indices. Thus, this process yields the interval modules of Definition~\ref{def:interval-module}. A full matrix reduction is provided in \cite{zomorodian2005,edelsbrunner2010}.

For $0\leqslant i\leqslant j\leqslant T$, let $r(i,j)=\rank(V_i\to V_j)$. Then the interval module $I[a,b]$ contributes one to $r(i,j)$ exactly when $a\leqslant i\leqslant j\leqslant b$. Hence, if $\mu(a,b)$ denotes the multiplicity of $I[a,b]$, then
\[ \mu(a,b)=r(a,b)-r(a-1,b)-r(a,b+1)+r(a-1,b+1), \]
where $r(-1,b)=0$ and $r(a,T+1)=0$. Consequently, the ranks of the composite maps determine every interval multiplicity. The matrix reduction gives existence, while the rank formula determines the multiplicities and hence uniqueness. Therefore, we have the following theorem:

\begin{theorem}[{\!\!\cite{zomorodian2005}}]\label{thm:persistence-decomposition}
Every pointwise finite-dimensional persistence module on a finite index set is a finite direct sum of interval modules.
The multiset of intervals is unique up to reordering.
\end{theorem}

The stability theorem for persistence diagrams under small perturbations of a filtering function is proved in \cite{cohensteiner2007}.
It controls geometric perturbations under its assumptions, and does not establish the semantic correctness of capability labels used below.

\begin{definition}\label{def:zigzag-module}\rm \
\begin{itemize}
  \item[(1)] Suppose adjacent spaces are related by inclusions, but the direction can be chosen at every step.
The diagram
\[ K_0\longleftrightarrow K_1\longleftrightarrow\cdots\longleftrightarrow K_T \]
induces a representation of a quiver of type $\mathbb{A}_{T+1}$ after homology is applied,
and we call this representation a \defines{zigzag persistence module}.

  \item[(2)]
Let $V=(V_i,\phi_i)$ and $W=(W_i,\psi_i)$ have the same arrow directions.
Their direct sum is defined by $(V\oplus W)_i=V_i\oplus W_i$ and structure maps $\phi_i\oplus\psi_i$.
A nonzero zigzag module is \defines{indecomposable} if it is not a direct sum of two nonzero zigzag modules.
Krull--Schmidt uniqueness \cite{carlsson2010} means that the isomorphism types and multiplicities of indecomposable factors are unique.

  \item[(3)] For an integer interval $[a,b]\subseteq\{0,\ldots,T\}$, the zigzag interval module $I[a,b]$ is $\FF$ at the indices in $[a,b]$ and zero elsewhere. Every arrow whose two endpoints lie in $[a,b]$ is the identity map in the prescribed arrow direction, and every other structure map is zero.
\end{itemize}
\end{definition}

\begin{theorem}\label{thm:zigzag-decomposition}
Every finite-dimensional zigzag persistence module over $\FF$ is a direct sum of interval modules,
and the interval multiset is unique up to reordering \cite{carlsson2010}.
Here, ordinary persistence is the special case in which all arrows point to the right.
\end{theorem}

\subsection{Large models}

We now define the terms from language modeling. The explanations are intentionally concrete, and they are included for readers who do not work in machine learning.

\begin{definition}\label{def:tokenizer}\rm
A \defines{tokenizer} is a map
\[ \operatorname{tok}:\Sigma^*\longrightarrow V^*, \]
where:
\begin{itemize}
  \item $\Sigma$ is a \defines{finite character alphabet};
  \item $\Sigma^*$ is the set of \defines{finite strings over $\Sigma$};
  \item and $V=\{v_1,\ldots,v_m\}$ is a finite \defines{token vocabulary}.
\end{itemize}
For a string $x$, the ordered list $\operatorname{tok}(x)=(t_1,\ldots,t_n)$ is its \defines{token sequence}.
\end{definition}

The alphabet $\Sigma$ contains written characters which can occur in raw text.
The vocabulary $V$ is the finite collection of pieces recognized by the model.
A token may be a whole word, part of a word, one Chinese character, punctuation, or a special control symbol.
Tokenization means that the written string is divided and replaced by these pieces.
For example, the phrase ``persistent homology'' may become two whole-word tokens or several subword tokens.
The tokenizer depends on the chosen vocabulary and algorithm \cite{sennrich2016}.
A token sequence is the ordered list which the model reads.
It is comparable to a sentence written with model-recognized pieces, but it need not agree with the grammatical division into words.
In an implementation, every token is stored by an integer index.

\begin{definition}\label{def:autoregressive-lm}\rm
An \defines{autoregressive language model} with parameters $\theta$ assigns a conditional distribution
\[
p_\theta(\,\cdot\mid t_1,\ldots,t_j)\in\Delta(V)
\]
to every token prefix $(t_1,\ldots,t_j)$.
For a sequence $(t_1,\ldots,t_n)$, its likelihood is
\[
p_\theta(t_1,\ldots,t_n)=\prod_{j=1}^n p_\theta(t_j\mid t_1,\ldots,t_{j-1}).
\]
Let the training corpus consist of token sequences $(t_{i,1},\ldots,t_{i,n_i})$ for $1\leqslant i\leqslant N$.
Training often minimizes the empirical negative log-likelihood
\begin{align}\label{eq:nll}
\mathcal L(\theta)&=-\frac1N\sum_{i=1}^N \sum_{j=1}^{n_i}\log p_\theta(t_{i,j}\mid t_{i,<j}).
\end{align}
\end{definition}

A language model predicts the next token from an existing token sequence, outputs a probability distribution, and may sample a new token during generation. During training, an optimization algorithm adjusts the neural network parameters through a loss function so that the probability of the true next token is as high as possible.
To be more precise, the model receives an existing token sequence and predicts the next token.
Its output assigns a probability to every item in the vocabulary $V$.
By the definition of a probability simplex, these probabilities are nonnegative and sum to $1$,
so the output lies in $\Delta(V)$. Then, a token is sampled from this distribution and appended to the sequence.
Repeating this operation allows the model to generate text one token at a time.
The parameter vector $\theta$ contains the numerical weights of the neural network.
The loss in \eqref{eq:nll} penalizes the model when it assigns low probability to the true next token.
Therefore, minimizing this loss adjusts $\theta$ so that the probability of the true next token is as high as possible.

A Transformer is a neural-network architecture which uses attention to combine information from different positions of a token sequence \cite{vaswani2017}.
In this article, a \defines{large model} means a Transformer-based language model with many parameters and broad training data.
The adjective ``large'' has no universal mathematical cutoff.
Our topology does not depend on a particular parameter count.

\begin{definition}\label{def:training-trajectory}\rm
A \defines{training trajectory} is a finite sequence of parameter vectors
\[ \theta_0,\theta_1,\ldots,\theta_T. \]
The model stored at $\theta_t$ is a \defines{checkpoint}.
If training continues from $\theta_t$ on a new data set or objective, the process is called \defines{fine-tuning}.
A decrease of earlier performance after later training is called \defines{catastrophic forgetting}.
\end{definition}

A checkpoint is a saved copy of the model at one moment during training.
It is not a new model. Indeed, it is the same model at a different stage.
The model is controlled by a very large collection of numbers.
These numbers are usually called weights or parameters.
Training is the process of adjusting these numbers.
Saving the current numbers gives a checkpoint.
Fine-tuning means that we continue training from an already trained model instead of starting from the beginning.
The new training stage may use new examples or a new objective.
It can teach the model a new subject, a new writing style, or a new way of responding to instructions.
Instruction tuning is a special case of fine-tuning.
In instruction tuning, the training examples are pairs of an instruction and a desired response.
The model therefore learns to follow human requests more reliably \cite{ouyang2022}.
Continued training can also have a negative effect.
The model may become worse at something it could do before.
For example, after learning to follow instructions, it might perform less well on an earlier mathematical task.
This loss of previously learned ability is called catastrophic forgetting \cite{luo2023}.

\begin{definition}\label{def:probe}\rm
A \defines{task} $q$ is a rule which assigns desired outputs to a class of inputs.
A \defines{prompt} is a token sequence supplied to a model as an instruction or context.
A \defines{probe set} $D_q=\{(x_i,y_i)\}_{i=1}^{N_q}$ is a finite scored sample for $q$.
For a checkpoint $\theta_t$, a \defines{behavioral score} is a number
\[ s_t(q;D_q)\in[0,1]. \]
It summarizes performance on the probe set.
\end{definition}

A task may ask the model to add two integers, translate a sentence, or answer a factual question.
The prompt is the actual text shown to the model.
The probe set is a finite examination paper with a scoring rule.
The score may be exact-match accuracy or another preregistered statistic.
Prompt wording can affect the result.
We therefore use several prompt paraphrases and random seeds in later definitions.

\begin{definition}\label{def:hidden-representation}\rm
For an input $x$, a checkpoint $t$, and a layer $\ell$, let $h_{t,\ell}(x)\in\mathbb R^d$
be the corresponding \defines{hidden representation}.
For fixed inputs $x_1,\ldots,x_N$, the set $\{h_{t,\ell}(x_i)\}_{i=1}^N$ is a point cloud in $\mathbb R^d$.
\end{definition}

A hidden representation is the internal coordinate record produced by a single layer.
It is neither a word nor a probability distribution.
Vectors that are close to one another are often interpreted as encoding similar internal states.
This interpretation, however, must be tested for the model and task under study.
The resulting finite collection of vectors can be analyzed using the point-cloud construction in Example \ref{ex:point-cloud}.

\section{Geometric descriptions of large models and their persistent homology}\label{sec:model}

\textsl{This section builds a geometric model linking large-model measurements to persistent homology.
It defines knowledge spaces and the cellular knowledge-state model, using cells of a finite regular CW complex and treating admissible states as subcomplexes.
It specializes this to large-model capability spaces, where probe scores and a reliability threshold yield a capability complex per checkpoint.
A theorem (see Theorem \ref{thm:single-cell}) shows that attaching one $n$-cell can only create a class in $\mathrm H_n$ or kill one in $\mathrm H_{n-1}$.
Since capability complexes along training need not be nested, the section uses union and intersection bridges to form zigzag persistence modules, distinguishes checkpoint, transition, and bridge-sensitive classes, and computes a finite four-stage example.
It ends with a comparison of ordinary and zigzag persistence and reporting guidelines.}

\subsection{Persistent homology from knowledge learning to capability spaces}

\begin{definition}\label{def:knowledge-space}\rm
Let $Q$ be a finite \defines{problem domain}.
Its elements are elementary questions or knowledge items.
A \defines{knowledge state} is a subset $S\subseteq Q$.
A family $\mathcal K\subseteq 2^Q$ is a \defines{knowledge space} if:
\begin{itemize}
  \item $\varnothing,Q\in\mathcal K$;
  \item (union axiom) $\bigcup\limits_{\alpha\in A}S_\alpha\in\mathcal K$ for every family $\{S_\alpha\}_{\alpha\in A}\subseteq\mathcal K$.
\end{itemize}
Here, each set lying in $\mathcal K$ is called an \defines{admissible state}, see \cite{doignon1985,doignon1999}.
\end{definition}

The set $Q$ is a finite syllabus. A state $S$ lists the items mastered by one learner.
The union axiom says that the union of admissible states is admissible.
It does not say that every subset is a possible state. Prerequisites may exclude some subsets.

\begin{example} \rm
Let $Q=\{a,b\}$, where $a$ is a prerequisite for $b$. Then the possible knowledge space is
$\mathcal K = \{\varnothing, \{a\}, \{a,b\}\} \subseteq 2^Q$.
Note that $\{b\}\notin\mathcal K$, since the prerequisite $a$ is missing.
We check the union-closure property as: $\varnothing \cup \{a\} = \{a\} \in \mathcal K$,
$\{a\} \cup \{a,b\} = \{a,b\} \in \mathcal K$, and $\varnothing \cup \{a,b\} = \{a,b\} \in \mathcal K$.
Thus the condition is satisfied.
\end{example}

\begin{definition}\label{def:ambient-complex}\rm
An \defines{ambient knowledge complex} is a finite regular CW complex $K^*$ with a \defines{labeling map}
\[ \lambda:\Cell(K^*)\longrightarrow\mathcal L. \]
Here, $\mathcal L$ is a set of semantic labels.
Labels of $0$-cells represent atomic items.
Labels of $1$-cells represent pairwise inference paths or relations.
Labels of higher cells represent preregistered higher-order synthesis.
\end{definition}

The ambient knowledge complex $K^*$ is a finite regular CW complex.
Regularity means that the closure of each cell is homeomorphic to a ball
and that its boundary is formed by lower-dimensional cells.
This regular cell structure allows the complex to be built dimension by dimension.
On this basis, the labeling map $\lambda$ assigns a semantic label to each cell.
Labels of $0$-cells represent atomic items, labels of $1$-cells represent pairwise inference paths or relations,
and labels of higher cells represent preregistered higher-order synthesis.
Therefore, the complex $K^*$ describes all encoded relations.
It is an ambient library rather than the current state of one learner.
It is worth noting that ordinary homology only concerns topological structure and forgets the specific semantics of each cell.
The labeling map is precisely used to store this semantic information forgotten by homology.
The significance of the regularity assumption lies precisely in this.
It provides a regular cellular foundation for the labeling map and for the homology analysis above.

\begin{definition}\label{def:cks}\rm
A \defines{cellular knowledge-state model}, or \defines{CKS model}, is a triple $(K^*,\lambda,\mathcal S)$.
Here, $K^*$ and $\lambda$ are as in Definition~\ref{def:ambient-complex},
the set $\mathcal S$ is a family of CW subcomplexes of $K^*$,
and each $K_S\in\mathcal S$ represents one knowledge state.
The model satisfies the following conditions.
\begin{itemize}
\item A cell belongs to $K_S$ only if all cells in its boundary belong to $K_S$.
\item Every included cell has an operational mastery criterion fixed before data analysis.
\item Different labels may be placed on cells of the same dimension.
\end{itemize}
\end{definition}

A vertex is an atomic item, such as a basic concept or knowledge item.
An edge can record a mastered path or relation between two items.
A $2$-cell can record a higher-order synthesis that fills a loop formed by pairwise relations.
For example, if three knowledge items have pairwise inference paths among them,
a $2$-cell represents the synthesis of the three into a whole.
The subcomplex condition prevents a higher-order relation from appearing without its boundary evidence.
That is, one cannot master a higher-order synthesis from nothing but must first master all its constituent parts.
The label $\lambda$ preserves cognitive meaning.
Homology computation only concerns topological structure such as connectivity, loops, and holes,
and it ignores specific semantics. The label retains what knowledge content a cell represents.

In summary, the CKS model formalizes knowledge states in topological language.
The complex $K^*$ is the ambient complex of all legal knowledge relations.
Each subcomplex $K_S$ in $\mathcal S$ is the knowledge state of one learner.
The subcomplex condition ensures that knowledge mastery respects prerequisites encoded by cell boundaries.
The labeling map preserves semantics so that homology computation focuses on structure while cognitive meaning is not lost.

\begin{proposition}\label{prop:knowledge-vs-betti}
Neither the number of cells nor the Betti numbers determine the other.
In particular, more mastered cells need not produce a larger total Betti number.
\end{proposition}

\begin{proof}
Let $K$ be one vertex and let $L$ be a tree with many vertices and edges.
Since both spaces are connected and contractible, we have
$\beta_0(K)=\beta_0(L)=1$ and $\beta_i(K)=\beta_i(L)=0$ for $i\geqslant1$.
However, their cell counts can be arbitrarily different.

For the converse, let $K'$ be a filled triangle and let $L'$ be the disjoint union of the boundary of a triangle and one isolated vertex.
The complex $K'$ has three vertices, three edges, and one two-cell, while $L'$ has four vertices and three edges. Thus both complexes have seven cells.
However, $K'$ is connected and contractible, whereas $\beta_0(L')=2$ and $\beta_1(L')=1$. Hence the number of cells does not determine the Betti numbers.
Cell count measures amount under the chosen encoding,
but Betti numbers measure organization and holes.
\end{proof}

\begin{definition}\label{def:relative-homology}\rm
For a subcomplex $K\subseteq L$, the \defines{relative $n$-chain group} is defined as the quotient $C_n(L,K;\FF) = C_n(L;\FF)/C_n(K;\FF)$.
The relative chain complex is
\[ C_{\bullet}(L,K;\FF) = (C_n(L,K;\FF),\partial_n^{\mathrm{rel}} )_{n\in\mathbb Z}.\]
Its homology is $\mathrm H_n(L,K;\FF)$.
The short exact sequence
\[0 \longrightarrow C_{\bullet}(K;\FF) \longrightarrow C_{\bullet}(L;\FF) \longrightarrow C_{\bullet}(L,K;\FF) \longrightarrow 0 \]
of chain complexes induces the long exact sequence
\begin{align}\label{eq:les}
\cdots\to\mathrm H_n(K)\to\mathrm H_n(L)\to\mathrm H_n(L,K)
&\xrightarrow{\delta}\mathrm H_{n-1}(K)\to\mathrm H_{n-1}(L)\to\cdots.
\end{align}
\end{definition}

\begin{theorem}\label{thm:single-cell}
Let $f: \partial \mathbb D^n = \mathbb S^{n-1} \to K$ be a continuous map from $\mathbb S^{n-1}$ to a finite CW complex $K$,
and $K'=K\cup_f e^n := (K \sqcup \mathbb D^n)/\sim$ be obtained from $K$ by attaching one $n$-cell,
where for each $x\in \mathbb S^{n-1}$, we define $x\sim f(x)$.
Assume $n\geqslant1$ and use coefficients in $\FF$.
Then only $\mathrm H_n$ and $\mathrm H_{n-1}$ can change. Let
\[ \delta:\mathrm H_n(K',K;\FF)\cong\FF\longrightarrow\mathrm H_{n-1}(K;\FF) \]
be the connecting map and let $r=\rank\delta\in\{0,1\}$. Then we have
\begin{align}\label{eq:bet-cha}
\beta_n(K')&=\beta_n(K)+1-r, \quad \beta_{n-1}(K')=\beta_{n-1}(K)-r,
\end{align}
and all other Betti numbers remain unchanged.
\end{theorem}

\begin{proof}
The relative cellular chain complex $C_{\bullet}(K',K;\FF)$ has one basis vector in degree $n$ and zero in every other degree.
Thus, we have
\[ \mathrm H_i(K',K;\FF)\cong \begin{cases}\FF,&i=n,\\0,&i\neq n.\end{cases}\]
Applying the long exact sequence \eqref{eq:les} from Definition \ref{def:relative-homology},
we obtain that, for $i\notin\{n,n-1\}$, both adjacent relative groups vanish.
The map $\mathrm H_i(K)\to\mathrm H_i(K')$ is therefore an isomorphism.

The relevant part is
\[
0 \to \mathrm H_n(K) \xrightarrow{\spadesuit} \mathrm H_n(K')
\xrightarrow{\heart} \FF
\xrightarrow{\clubsuit} \mathrm H_{n-1}(K)
\xrightarrow{\diamo} \mathrm H_{n-1}(K')\to0.
\]
Here $\clubsuit = \delta$ is the connecting map given in \eqref{eq:les}; in particular, $\rank(\clubsuit)=r$.
Exactness shows that $\spadesuit$ is injective, and it also shows that $\Ima(\heart)=\Ker(\clubsuit)$.
Thus, the quotient $\mathrm H_n(K')/\Ima(\spadesuit)\cong\Ima(\heart)=\Ker(\clubsuit)$ has dimension $\dim\Ker(\delta)=1-r$ ($\leqslant\dim\FF=1$).
This proves the first formula given in \eqref{eq:bet-cha}.
Exactness also shows that $\diamo$ is surjective and $\Ker(\diamo)=\Ima(\clubsuit)$.
Therefore, $\mathrm H_{n-1}(K') = \Ima(\diamo) \cong \mathrm H_{n-1}(K)/\Ker(\diamo)$,
i.e., it is the quotient of $\mathrm H_{n-1}(K)$ by an $r$-dimensional subspace.
Then the second formula given in \eqref{eq:bet-cha} holds.
\end{proof}

By Theorem \ref{thm:single-cell}, we obtain the facts that
if $r=0$, then the attached cell creates one class in $\mathrm H_n$,
and if $r=1$, then it kills one class in $\mathrm H_{n-1}$.
An attached edge can create a loop or merge two connected components,
and an attached two-cell can fill a one-dimensional loop.
Theorem \ref{thm:single-cell} is algebraic. Its cognitive meaning depends on the preregistered label of the cell.

\begin{corollary}\label{cor:new-item}
Suppose a new atomic item $X$ is inserted as a vertex and two new edges attach it to vertices $u$ and $v$ in the same connected component of $K$.
Assume the first edge is attached before the second.
The insertion of the new vertex creates one class in $\mathrm H_0$, and the first edge kills this class. Hence these two operations together have no net effect on homology.
The second edge creates one class in $\mathrm H_1$. If a two-cell is subsequently attached along the resulting loop, that two-cell kills the new class.
\end{corollary}

\begin{proof}
The new vertex first creates one component.
The first edge joins this component to the component containing $u$.
The connecting map in Theorem~\ref{thm:single-cell} has rank one, so the new $\mathrm H_0$ class is killed.
The second edge has both endpoints in the same component because $u$ and $v$ were connected in $K$.
Its connecting map to $\mathrm H_0$ is zero.
Theorem~\ref{thm:single-cell} with $n=1$ shows that one $\mathrm H_1$ class is created.
If a two-cell is then attached along this loop, the same theorem with $n=2$ shows that the class can be killed.
\end{proof}

\begin{example}\label{ex:knowledge-x}\rm
Let $A$ be an initial state and let $B$ be newly mastered material.
Assume that neither $A$ nor $B$ alone supports a derivation of $X$.
Set $K_0=K_{A\cup B}$, and choose vertices $u$ and $v$ in the same connected component of $K_0$.
After $A\cup B$ is mastered, suppose that the operational mastery criterion for the new atomic item $X$ is satisfied. Attach a new $0$-cell $x_X$ labeled by $X$ and set $K_1=K_0\sqcup\{x_X\}$.
Next attach a relation edge $e_u^1$ from $x_X$ to $u$ and then a relation edge $e_v^1$ from $x_X$ to $v$. Set $K_2=K_1\cup e_u^1$ and $K_3=K_2\cup e_v^1$.
If the model also records a higher-order synthesis which fills the resulting loop, attach a $2$-cell along that loop and denote the resulting complex by $K_4$.
This is precisely the process in Corollary~\ref{cor:new-item}. The label $X$ belongs only to the $0$-cell $x_X$; the edges and the optional $2$-cell have separate relation or synthesis labels.
Theorem~\ref{thm:single-cell} determines the possible homological change at each attachment. Whether an attachment is a candidate topological emergence is determined by its induced homology map as in Definition~\ref{def:three-level}, rather than merely by the appearance of the label $X$.
Many new facts can be leaves of a knowledge tree.
They enlarge the state without changing $\mathrm H_1$.
\end{example}

\begin{definition}[{Emergence}]\label{def:three-level}\rm \ 
\begin{itemize}
\item[(1)] A \defines{candidate cell event} occurs when a preregistered mastery rule changes the included cell set.
\item[(2)] Fix a homological degree $k$ and an arrow $f:A\to B$ that represents the event in a preregistered persistence or zigzag diagram. A \defines{candidate topological emergence} occurs when the induced map $\mathrm H_k(f):\mathrm H_k(A;\FF)\to\mathrm H_k(B;\FF)$ is not an isomorphism.
\item[(3)] A \defines{candidate capability emergence} occurs when the candidate topological emergence satisfies an experiment-specific robustness and null-model validation rule fixed before data analysis and is accompanied by improvement on an independent held-out transfer task.
\end{itemize}
\end{definition}

The present article does not prescribe a universal statistical test for robustness or null-model validation. In each empirical application, the researcher must specify and preregister the barcode discrepancy statistic, the finite threshold grid, the resampling or permutation unit, the null hypothesis, the null model, the rejection bound, and any multiple-testing correction. These experiment-specific choices form part of the operational definition of candidate capability emergence for that application.

\def\median{\mathop{\mathrm{median}}}

We now apply this construction to large models.
Fix atomic tasks $Q=\{q_1,\ldots,q_m\}$ and a finite regular CW complex $K^*$ whose cells have capability labels.
For every cell $\sigma$, choose a probe set $D_\sigma$, a finite nonempty set $\mathcal P_\sigma$ of prompt paraphrases, a finite nonempty set $\mathcal R_\sigma$ of random seeds, and a scoring function with values in $[0,1]$.
For every pair $(p,r)\in\mathcal P_\sigma\times\mathcal R_\sigma$,
the value $s_t(\sigma;D_\sigma,p,r)$ denotes the score obtained from the probe set $D_\sigma$,
the prompt paraphrase $p$, and the random seed $r$.
Consider the finite multiset
\[ \mathcal A_t(\sigma) = \{s_t(\sigma;D_\sigma,p,r): (p,r)\in\mathcal P_\sigma\times\mathcal R_\sigma \}\checks{.} \]
The purpose of using multiset here is to ensure that the repeated values are retained with their multiplicities.
Let $N =|\mathcal P_\sigma|\,|\mathcal R_\sigma|$, and arrange the elements of $\mathcal A_t(\sigma)$ as
$a_1\leqslant a_2\leqslant\cdots\leqslant a_N$.
We use the notation $\median\limits_{ p\in\mathcal P_\sigma, r\in\mathcal R_\sigma } s_t(\sigma;D_\sigma,p,r) $
for the number
\[ \median\limits_{ p\in\mathcal P_\sigma, r\in\mathcal R_\sigma } s_t(\sigma;D_\sigma,p,r)
:= \begin{cases}
     a_{\frac{N+1}{2}}, & \text{if $N$ is odd},\\
     \dfrac{a_{\frac{N}{2}}+a_{\frac{N}{2}+1}}{2}, & \text{if $N$ is even}.
   \end{cases}\]
At checkpoint $t$, define
\begin{align}\label{eq:cel-sco}
s_t(\sigma) = \median_{
  p\in\mathcal P_\sigma,
  r\in\mathcal R_\sigma
  } s_t(\sigma;D_\sigma,p,r).
\end{align}

For cells $\tau$ and $\sigma$ of $K^*$, we write $\tau\le \sigma$ when $\tau\subseteq\overline{\sigma}$. Thus the face relation used below is reflexive, and $\sigma\le \sigma$.

\begin{definition}\label{def:capab comp}\rm
For a threshold $\theta\in[0,1]$, the \defines{capability complex} at checkpoint $t$ is
\[ K_t(\theta)=\{\sigma\in\Cell(K^*)\mid s_t(\tau)\geqslant\theta\text{ for every face }\tau\le\sigma\}. \]
The right side denotes the CW subcomplex formed by these cells.
\end{definition}

A vertex probe tests one atomic capability.
An edge probe tests a relation or transfer between two atomic capabilities.
A higher-cell probe tests a preregistered joint behavior which cannot be reduced to one boundary item.
The face condition prevents a high-dimensional cell from entering when its required lower-dimensional evidence fails.
This construction turns AI measurements into a finite geometric object.

\begin{lemma}\label{lem:capability-subcomplex}
The set $K_t(\theta)$ in Definition~\ref{def:capab comp} is a CW subcomplex of $K^*$.
\end{lemma}

\begin{proof}
Take $\sigma\in K_t(\theta)$ and let $\rho$ be a face of $\sigma$.
Every face $\tau$ of $\rho$ is also a face of $\sigma$.
The defining condition for $\sigma$ shows that $s_t(\tau)\geqslant\theta$.
Hence $\rho\in K_t(\theta)$.
The cell set is closed under faces.
It therefore forms a CW subcomplex.
\end{proof}

\begin{proposition}\label{prop:threshold-filtration}
For a fixed checkpoint $t$, if $\theta_1\geqslant\theta_2$, then
$K_t(\theta_1)\subseteq K_t(\theta_2)$.
Thus decreasing the threshold produces an ordinary filtration.
\end{proposition}

\begin{proof}
If $\sigma\in K_t(\theta_1)$, every face $\tau\le\sigma$ satisfies $s_t(\tau)\geqslant\theta_1$.
Since $\theta_1\geqslant\theta_2$, the same scores satisfy $s_t(\tau)\geqslant\theta_2$.
Definition~\ref{def:capab comp} shows that $\sigma\in K_t(\theta_2)$.
\end{proof}

\begin{definition}\label{def:capab-space}\rm
For a fixed checkpoint $t$, let $\mathcal K_t=\{K_t(\theta)\}_{\theta\in[0,1]}$.
Together with the inclusion maps $K_t(\theta_1)\hookrightarrow K_t(\theta_2)$ for $\theta_1\geqslant\theta_2$, the family $\mathcal K_t$ is called the \defines{capability space} at checkpoint $t$. More precisely, $\mathcal K_t$ is a filtered capability space whose filtration parameter is the reliability threshold.
\end{definition}

Definition~\ref{def:capab comp} and Proposition~\ref{prop:threshold-filtration} define a TDA construction from behavioral data. The finite score table is the observed data. The face condition converts these scores into the subcomplex $K_t(\theta)$. The threshold orders these subcomplexes into a filtration. Applying $\mathrm H_k(-;\FF)$ and Theorem~\ref{thm:persistence-decomposition} produces a barcode. This construction differs from distance-based point-cloud TDA in its rule from data to a complex. Its barcode describes the chosen organization of capabilities. It does not directly measure the number of mastered tasks.

Scores should be estimated on held-out data.
Bootstrap intervals can quantify finite-sample uncertainty \cite{efron1993}.
Multiple tests should be corrected, for example by false discovery rate control \cite{benjamini1995}.
Cell labels, prompts, thresholds, bridge rules, and null models should be fixed before the barcode is inspected.
If these assumptions fail, homology remains computable but its intended capability meaning and its statistical guarantees may fail.

For fixed $t$, a decreasing threshold provides ordinary persistent homology.
This check is useful because a discrete score can make a smooth behavioral change appear sudden \cite{schaeffer2023}.
The capability space $\mathcal K_t$ from Definition~\ref{def:capab-space} has a monotone threshold direction, while the training direction is generally nonmonotone.
We analyze one-dimensional slices.
Classical multiparameter persistence assumes monotonicity in every parameter direction \cite{carlsson2009multi}.
It is not identical to the present family.
Quasi-zigzag persistence combines scale and temporal zigzags \cite{dey2026}.
We use preregistered one-dimensional slices to keep the behavioral meaning transparent.

An alternative geometry uses hidden representations.
For fixed inputs, checkpoint $t$, and layer $\ell$, set
\[
Z_{t,\ell}=\{h_{t,\ell}(x_i)\mid1\leqslant i\leqslant N\}\subset\mathbb R^d.
\]
Vietoris--Rips, \v Cech, or nearest-neighbor complexes provide scale filtrations \cite{edelsbrunner2010}.
Layerwise topological changes in neural networks were studied in \cite{naitzat2020}.
The behavioral capability complex and the representation point cloud answer different questions.
The first records combinations of tasks that the model performs reliably.
The second records the geometric organization of hidden states.
Their barcodes should not be assigned the same semantic meaning without validation.

The paper therefore uses two data-to-topology maps. The behavioral map sends probe scores to the filtered capability complexes $\{K_t(\theta)\}_{\theta}$. The representation map sends activation vectors to metric point clouds and then to Vietoris--Rips or \v Cech filtrations. The first map studies relations among measured capabilities. The second map studies the geometry of hidden representations. Agreement between their barcodes is an empirical question and is not implied by persistent homology.

\subsection{Zigzag persistent homology for learning, forgetting, and recovery}

Ordinary persistence requires $K_t\subseteq K_{t+1}$.
The score in \eqref{eq:cel-sco} may rise or fall during training.
Continual fine-tuning has produced both new learning and loss of old knowledge under specific models and task orders \cite{luo2023}.
This observation motivates a nonnested model.
It does not show that every large-model training process is nonmonotone.
Two adjacent checkpoints may satisfy both
$K_t\setminus K_{t+1}\neq\varnothing$ and $K_{t+1}\setminus K_t\neq\varnothing$.

\begin{proposition}\label{prop:bridge-zigzag}
Let $K_0,\ldots,K_T$ be arbitrary subcomplexes of one finite CW complex $K$.
They produce two zigzags with canonical inclusions.
They are
\begin{align}\label{eq:uni-zig}
K_0&\longrightarrow K_0\cup K_1\longleftarrow K_1\longrightarrow K_1\cup K_2\longleftarrow K_2\longrightarrow\cdots
\end{align}
and
\begin{align}\label{eq:int-zig}
K_0&\longleftarrow K_0\cap K_1\longrightarrow K_1\longleftarrow K_1\cap K_2\longrightarrow K_2\longleftarrow\cdots.
\end{align}
After $\mathrm H_k(-;\FF)$ is applied, both finite-dimensional zigzag modules have unique interval decompositions.
\end{proposition}

\begin{proof}
We first verify that every bridge is a subcomplex.
If $\sigma\in\Cell(K_i\cup K_{i+1})$, then it belongs to one of the two subcomplexes.
Every face of $\sigma$ belongs to that same subcomplex and hence to the union.
If $\sigma\in\Cell(K_i\cap K_{i+1})$, each face belongs to both subcomplexes and hence to the intersection.
Thus unions and intersections are subcomplexes.
Lemma~\ref{lem:capability-subcomplex} verifies this fact directly for capability states.

For adjacent states, the relations
\[
K_i\subseteq K_i\cup K_{i+1}\supseteq K_{i+1}
\quad\text{and}\quad
K_i\supseteq K_i\cap K_{i+1}\subseteq K_{i+1}
\]
provide the arrows in \eqref{eq:uni-zig} and \eqref{eq:int-zig}.
Proposition~\ref{prop:homology-functoriality} induces linear maps in the same directions.
Every chain group is finite-dimensional because $K$ has finitely many cells.
Each homology group is also finite-dimensional by Definition~\ref{def:homology}.
Theorem~\ref{thm:zigzag-decomposition} now provides the unique interval decompositions.
\end{proof}

\begin{remark}\rm
Proposition~\ref{prop:bridge-zigzag} places a nonmonotone checkpoint sequence inside TDA. At one checkpoint, the threshold direction in Proposition~\ref{prop:threshold-filtration} is an ordinary filtration. Across checkpoints, additions and deletions may occur together, so the spaces need not form a filtration. The union or intersection bridges replace the missing inclusions by a zigzag diagram. Theorem~\ref{thm:zigzag-decomposition} then provides interval summaries. The bridge rule is part of the TDA construction because different bridges can induce different continuation maps.
\end{remark}

Let $Z_{2i}=K_i$ and let $Z_{2i+1}=B_{i,i+1}$, where the bridge is a union or an intersection.
Assign times $i$ and $i+\tfrac12$ to these nodes.
Barcode endpoints should first be reported on the original $2T+1$ indices.
An interval supported only on a bridge node is a \defines{transition class}.
It does not show that either checkpoint has the corresponding capability.
A union bridge may fill or create a loop by mixing cells from both sides.
An intersection bridge keeps only the common skeleton.
If the two bridge schemes produce different continuations, we call the interval \defines{bridge-sensitive}.
We do not directly interpret it as a capability death.
Both schemes should be reported as a sensitivity analysis.

\begin{example}\rm
We compute a finite example.
Let the atomic capabilities be vertices $a,b,c$.
Let $L_0$ contain oriented edges $u=(a,b)$ and $v=(b,c)$.
With vertex basis $(a,b,c)$ and edge basis $(u,v)$, we have
\[
\partial_1^{L_0}=\begin{pmatrix}-1&0\\1&-1\\0&1\end{pmatrix},
\qquad \rank\partial_1^{L_0}=2.
\]
The space $L_0$ is connected and $(\beta_0,\beta_1)=(1,0)$.
Adding $x=(c,a)$ to obtain $L_1$, then we have
\[
\partial_1^{L_1}=\begin{pmatrix}-1&0&1\\1&-1&0\\0&1&-1\end{pmatrix},
\quad \rank\partial_1^{L_1}=2,
\quad \mathrm H_1(L_1)=\FF\langle u+v+x\rangle.
\]
The orientation only fixes the cellular bases and the signs in the matrix,
and ordinary CW homology does not distinguish ``$a$ implies $b$'' from ``$b$ implies $a$''.
See the top row shown in \Pic \ref{fig:finite-zigzag}.
An inference direction can be stored in $\lambda$, but it does not enter our homology calculation.
\begin{figure}[htbp]
\centering
\begin{tikzpicture}[
  vertex/.style={circle,fill=black,inner sep=1.6pt},
  edge/.style={->,>=stealth,line width=0.75pt},
  every node/.style={font=\small}
]
\begin{scope}[xshift=0cm,yshift=0cm]
  \coordinate (l0a) at (-0.9,0);
  \coordinate (l0b) at ( 0.0,0);
  \coordinate (l0c) at ( 0.9,0);
  \draw[edge] (l0a)--node[above] {$u$} (l0b);
  \draw[edge] (l0b)--node[above] {$v$} (l0c);
  \foreach \z/\lab in {l0a/$a$,l0b/$b$,l0c/$c$}{\node[vertex] at (\z) {};\node[below=3pt] at (\z) {\lab};}
  \node at (0,0.85) {$L_0$};
  \node at (0,-0.72) {$\beta_1=0$};
\end{scope}

\node at (2.0,0) {$\hookrightarrow$};

\begin{scope}[xshift=4cm,yshift=0cm]
  \coordinate (l1a) at (-0.85,-0.40);
  \coordinate (l1b) at ( 0.00,0.65);
  \coordinate (l1c) at ( 0.85,-0.40);
  \draw[edge] (l1a)--node[above left] {$u$} (l1b);
  \draw[edge] (l1b)--node[above right] {$v$} (l1c);
  \draw[edge] (l1c)--node[below] {$x$} (l1a);
  \foreach \z/\pos/\lab in {l1a/below left/$a$,l1b/above/$b$,l1c/below right/$c$}{\node[vertex] at (\z) {};\node[\pos=3pt] at (\z) {\lab};}
  \node at (0,1.28) {$L_1$};
  \node at (0,-1.05) {$\beta_1=1$};
\end{scope}

\node at (6.0,0) {$\hookleftarrow$};

\begin{scope}[xshift=8cm,yshift=0cm]
  \coordinate (l2a) at (-0.9,0);
  \coordinate (l2b) at (0,0);
  \coordinate (l2c) at (0.9,0);
  \draw[edge] (l2a)--node[above] {$u$} (l2b);
  \draw[edge] (l2b)--node[above] {$v$} (l2c);
  \foreach \z/\lab in {l2a/$a$,l2b/$b$,l2c/$c$}{\node[vertex] at (\z) {};\node[below=3pt] at (\z) {\lab};}
  \node at (0,0.85) {$L_2=L_0$};
  \node at (0,-0.72) {$\beta_1=0$};
\end{scope}

\node at (10.0,0) {$\hookrightarrow$};

\begin{scope}[xshift=12cm,yshift=0cm]
  \coordinate (l3a) at (-0.82,0);
  \coordinate (l3b) at (0,0.72);
  \coordinate (l3c) at (0.82,0);
  \coordinate (l3d) at (0,-0.72);
  \draw[edge] (l3a)--node[above left] {$u$} (l3b);
  \draw[edge] (l3b)--node[above right] {$v$} (l3c);
  \draw[edge] (l3c)--node[below right] {$p$} (l3d);
  \draw[edge] (l3d)--node[below left] {$q$} (l3a);
  \foreach \z/\pos/\lab in {l3a/left/$a$,l3b/above/$b$,l3c/right/$c$,l3d/below/$d$}{\node[vertex] at (\z) {};\node[\pos=3pt] at (\z) {\lab};}
  \node at (0,1.38) {$L_3$};
  \node at (0,-1.50) {$\beta_1=1$};
\end{scope}

\begin{scope}[xshift=4cm,yshift=-3.25cm]
  \coordinate (a1a) at (-0.85,-0.40);
  \coordinate (a1b) at (0,0.65);
  \coordinate (a1c) at (0.85,-0.40);
  \draw[edge] (a1a)--node[above left] {$u$} (a1b);
  \draw[edge] (a1b)--node[above right] {$v$} (a1c);
  \draw[edge] (a1c)--node[below] {$x$} (a1a);
  \foreach \z/\pos/\lab in {a1a/below left/$a$,a1b/above/$b$,a1c/below right/$c$}{\node[vertex] at (\z) {};\node[\pos=3pt] at (\z) {\lab};}
  \node at (0,1.28) {$L_1$};
  \node at (0,-1.08) {$\mathrm H_1(L_1)\cong\FF$};
\end{scope}

\node at (6,-3.25) {$\xrightarrow{\text{attach }f}$};

\begin{scope}[xshift=8cm,yshift=-3.25cm]
  \coordinate (f1a) at (-0.85,-0.40);
  \coordinate (f1b) at (0,0.65);
  \coordinate (f1c) at (0.85,-0.40);
  \fill[blue!18] (f1a)--(f1b)--(f1c)--cycle;
  \draw[edge] (f1a)--node[above left] {$u$} (f1b);
  \draw[edge] (f1b)--node[above right] {$v$} (f1c);
  \draw[edge] (f1c)--node[below] {$x$} (f1a);
  \foreach \z/\pos/\lab in {f1a/below left/$a$,f1b/above/$b$,f1c/below right/$c$}{\node[vertex] at (\z) {};\node[\pos=3pt] at (\z) {\lab};}
  \node at (0,0) {$f$};
  \node at (0,1.28) {$L_1^+$};
  \node at (0,-1.08) {$\mathrm H_1(L_1^+)=0$};
\end{scope}
\end{tikzpicture}
\caption{The top row shows the checkpoint zigzag $L_0\hookrightarrow L_1\hookleftarrow L_2\hookrightarrow L_3$;
and the bottom row shows the auxiliary attachment of the 2-cell $f$}
\label{fig:finite-zigzag}
\end{figure}
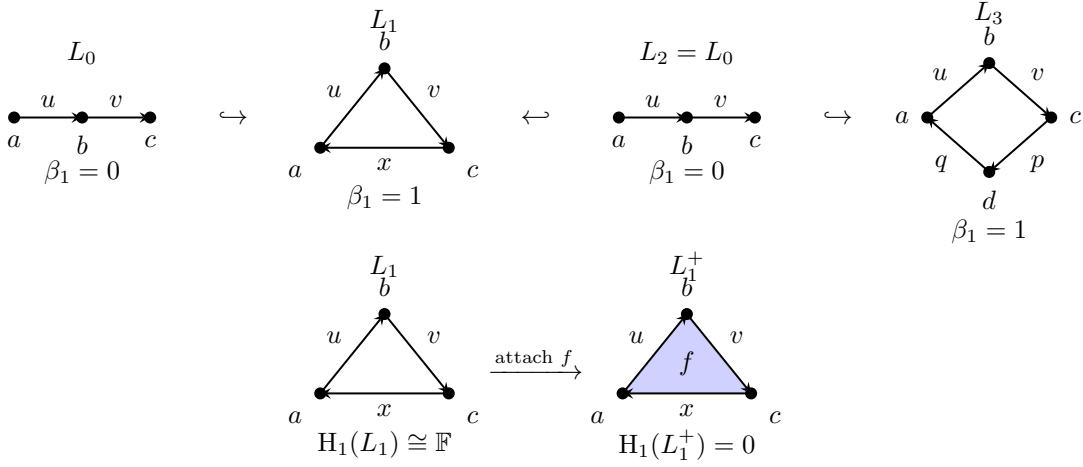

Attach a $2$-cell $f$ along the loop $u+v+x$ and call the result $L_1^+$.
Then
\[ \partial_2^{L_1^+}(f)=u+v+x,
  \quad [\partial_2]_{(u,v,x)}=(1,1,1)^{\mathsf T},
  \quad \beta_1(L_1^+)=0. \]
This is an algebraic example in which a $2$-cell kills a one-dimensional class.
A higher-order synthesis interpretation still requires preregistered semantics and external behavioral evidence.
Set $L_2=L_0$ by deleting $x$ (see the top row shown in \Pic \ref{fig:finite-zigzag}).
Then we have $\mathrm H_1(L_2)=0$.
Adding a vertex $d$ and edges $p=(c,d)$ and $q=(d,a)$ to obtain $L_3$,
then we have that this connected graph has four vertices and four edges.
Hence $\rank\partial_1^{L_3}=3$ and
$\mathrm H_1(L_3)=\FF\langle u+v+p+q\rangle$.
Thus, we have
\[ L_0\hookrightarrow L_1\hookleftarrow L_2\hookrightarrow L_3,
\quad 0\longrightarrow\FF\longleftarrow0\longrightarrow\FF \cong I[1,1]\oplus I[3,3]. \]
Algebraically, the classes of $L_1$ and $L_3$ belong to different intervals.
Calling the latter a recovery instead of a new class requires independent behavioral evidence.
A union bridge may continue a class even when its cycle representative changes.
The precise statement is only that the homology class continues through the chosen bridge.
\end{example}

Birth and death must be interpreted together with cell types.
Deleting a critical path may kill an $\mathrm H_1$ class and represent forgetting.
Attaching a $2$-cell may kill an $\mathrm H_1$ class and represent higher-order synthesis.
The two events have the same effect on the Betti number.
They differ in their update mechanisms and relative homology.
One should store the dimension, attaching map, and behavioral score of every update.

Nearest-neighbor complexes from $Z_{t,\ell}$ are also generally not nested in $t$ or $\ell$.
Unions or intersections can be placed between adjacent point-cloud complexes.
Zigzag persistence has been used along network depth to study last-token representations of large language models \cite{gardinazzi2026}.
Persistent homology has also been tracked across dense fine-tuning checkpoints \cite{malhotra2026}.
We do not claim that the use of training checkpoints alone is new.
Our proposal combines a preregistered behavioral capability complex, a nonnested training zigzag, and independent behavioral validation.

\subsection{Comparison between ordinary persistence and zigzag persistence}

\begin{table}[htbp]
\centering\small
\renewcommand{\arraystretch}{1.35}
\begin{tabular}{p{2.4cm}p{5.1cm}p{5.1cm}}
\toprule
Item & Ordinary persistence & Zigzag persistence\\
\midrule
Space relation & $K_0\subseteq K_1\subseteq\cdots$ & Each adjacent arrow may point forward or backward\\
Suitable process & Monotone accumulation, scale expansion, and knowledge without deletion & Learning, forgetting, recovery, and representation replacement\\
Algebraic object & Persistence module with all arrows in one direction & Representation of an arbitrarily oriented quiver of type $A_n$\\
Barcode meaning & Birth and death in a monotone filtration & Support interval in a nonmonotone sequence\\
Main advantage & Simple construction and mature stability theory & Deletion is represented directly by reversed inclusions\\
Main limitation & Deletion cannot be expressed by forward inclusions & Continuation depends on the union or intersection bridge\\
\bottomrule
\end{tabular}
\caption{The roles of ordinary and zigzag persistence in the present model \cite{edelsbrunner2002,zomorodian2005,carlsson2010}.}
\label{tab:comparison}
\end{table}

\begin{proposition}\label{prop:history-union}
Let $K_0,\ldots,K_T$ be subcomplexes of one ambient complex and let
$\widetilde K_t=\bigcup_{s\leqslant t}K_s$.
Then $\widetilde K_t\subseteq\widetilde K_{t+1}$.
If $\sigma\in\Cell(K_s)\setminus\Cell(K_t)$ for some $s<t$, then $\sigma\in\Cell(\widetilde K_t)$.
Thus the historical-union filtration cannot distinguish present existence from past existence.
\end{proposition}

\begin{proof}
Take $\sigma\in\Cell(\widetilde K_t)$.
There is $s\leqslant t$ such that $\sigma\in\Cell(K_s)$.
The same $K_s$ occurs in $\widetilde K_{t+1}$.
Hence $\sigma\in\Cell(\widetilde K_{t+1})$.
This proves the inclusion.

Now assume $\sigma\in\Cell(K_s)\setminus\Cell(K_t)$ for some $s<t$.
Since $s\leqslant t$, we have
\[ \sigma\in\Cell(K_s)\subseteq\Cell \bigg(\bigcup_{r\leqslant t}K_r\bigg)
 = \Cell(\widetilde K_t). \]
The cell has disappeared from the current state $K_t$, but it remains in the historical union.
The historical union cannot decide whether the cell is still present at time $t$.
Proposition~\ref{prop:bridge-zigzag} retains this information through a nonmonotone diagram.
\end{proof}

\begin{remark}\rm
If an ideal learner satisfies $K_t\subseteq K_{t+1}$, ordinary persistent homology is sufficient.
Old capabilities may decline during fine-tuning \cite{luo2023}.
The historical union answers what appeared at least once.
The zigzag sequence can still represent what the current model can do.
\end{remark}

An empirical study should report the following information.
\begin{itemize}
\item The checkpoints and their training-data stages.
\item The operational definitions of capability cells, labels, and probes.
\item Sensitivity analyses for the threshold $\theta$ and the bridge scheme.
\item Barcode changes together with continuous behavioral scores.
\end{itemize}
A homological change should be reported as a robust candidate topological emergence only when it satisfies the experiment-specific validation rule preregistered for that study.
It becomes a candidate capability emergence only if an independent transfer task changes at the same time.
The bottleneck stability theorem controls small perturbations when filtering functions are defined on a common triangulable space and the perturbation is bounded in the uniform norm \cite{cohensteiner2007}.
It does not automatically cover arbitrary perturbations of a training zigzag.
It does not prove that the capability complex is correct.
It also does not turn correlation into causation.

\section{Conclusions and limitations}

Ordinary persistent homology first describes a local algebraic effect in the CKS model.
The attachment of an $n$-cell can only create an $\mathrm H_n$ class or kill an $\mathrm H_{n-1}$ class.
Zigzag persistent homology then removes the monotonicity assumption.
It joins local events along a training process into a decomposable global algebraic object.
The amount of capability and the organization of capability can therefore be measured separately.
Events preregistered as learning, forgetting, recovery, or higher-order synthesis can be compared in one algebraic language.

The main limitations come from the encoding.
The ambient complex $K^*$, cell meanings, probe scores, and thresholds are chosen by the researcher.
The resulting barcode is conditional on these choices.
The bridge scheme can change interval continuation.
Finite samples can misclassify cells.
Classical stability results do not directly establish statistical stability for a training zigzag.
Our theorems explain how cellular updates change homology.
They do not prove that a homology class is a semantic capability.
They do not prove a causal mechanism for training-induced capability change.
Empirical studies should preregister construction rules, use held-out data, report resampling intervals, correct multiple tests, use permutation null models, and validate emergence on independent tasks.

Ordinary CW homology also forgets the direction of a relation.
If direction-sensitive invariants are required, one may use directed complexes, path homology, or quiver representations.
The regular CW assumption restricts the attaching maps.
A general closed walk which repeats a vertex or an edge is excluded here.
One can duplicate vertices or edges to construct a regular ambient complex, or one can use nonregular CW attachments.
Subdivision of the original graph alone does not solve every repeated-boundary problem.

\section*{Declarations}

\paragraph{Authors' Contributions}

The author contributed to the conception, methodology, derivation, and writing of this paper.

\paragraph{Competing Interests}

The authors declare that they have no conflicts of interest as defined by the journal, nor any other interests that could be perceived as influencing the results presented in this paper.

\paragraph{Data availability}
No external data set or computational code was used.

\paragraph{Ethical Approval}

This article does not require ethical approval.

\paragraph{Funding}
Yu-Zhe Liu is supported by
the National Natural Science Foundation of China (Grant Nos. 12401042 and 12561008),
the Guizhou Provincial Key Laboratory of Applied Mathematics and Computing Power \& Algorithms (No. Qiankehe Platform ZSYS[2025]039),
the Science and Technology Foundation of the Guizhou S~\&~T Department (Grant Nos. VZD[2026]001, ZD[2025]085 and ZK[2024]YiBan066),
and Scientific Research Foundation of Guizhou University (Grant No. [2023]16).

%

\bibliographystyle{abbrv}
\bibliography{persistent_homo_and_llm}

\end{document}